\documentclass[12pt, a4paper, reqno, twoside, makeidx]{amsart}
\usepackage[margin=2.8cm, marginpar=3cm]{geometry}

\usepackage{verbatim}
\usepackage{amsmath}
\usepackage{amsfonts}
\usepackage{amssymb}
\usepackage{graphicx}
\usepackage{color}
\usepackage{empheq}
\usepackage[font={small,it}]{caption}
\usepackage[all]{xy}
\usepackage{tikz-cd}
\usepackage{tikz}
\usetikzlibrary{calc}
\usepackage{amssymb}
\usepackage{amsmath, amsthm, amssymb, amsfonts, amscd}
\usepackage[english]{babel}
\usepackage[all]{xy}
\usepackage{url}
\usepackage{hyperref}
\usepackage{mathtools}
\usepackage{bm}
\usepackage{braket}

\theoremstyle{remark}
\theoremstyle{plain}
\newtheorem{thm}{Theorem}[section]
\newtheorem{lem}[thm]{Lemma}
\newtheorem{prop}[thm]{Proposition}
\newtheorem{defi}[thm]{Definition}
\newtheorem{cor}[thm]{Corollary}

\newtheorem{rem}[thm]{Remark}

\newcommand{\Z}{\mathbb{Z}}
\newcommand{\Q}{\mathbb{Q}}
\newcommand{\R}{\mathbb{R}}

\newcommand{\inv}{\tau}
\newcommand{\oinv}{\sigma}

\newcommand{\x}{\mathbf{x}}
\newcommand{\y}{\mathbf{y}}
\newcommand{\gr}{{\rm{gr}}}
\newcommand{\boa}{\bm{\alpha}}
\newcommand{\bob}{\bm{\beta}}

\newcommand{\s}{\mathfrak{s}}

\newcommand{\Sym}{{\rm{Sym}}}
\newcommand{\T}{\mathbb{T}}

\newcommand{\Spinc}{{\rm{Spin}^c}}

\newcommand{\Ker}{{\rm{Ker}}}
\newcommand{\Imm}{{\rm{Im}}}
\newcommand{\id}{{\rm{id}}}

\newcommand{\spinc}{\mbox{spin}^c}

\newcommand{\HFB}{{\rm {HFB}}^-}

\newcommand{\CFB}{{\rm {CFB}}^-}

\newcommand{\deltaover}{{\overline {\delta}}}
\newcommand{\deltaunder}{{\underline {\delta}}}
\newcommand{\Cone}{{\rm {Cone}}}

\newcommand{\Field}{{\mathbb {F}}}
\newcommand{\F}{{\mathbb {F}}}

\newcommand{\CFm}{{\rm {CF}}^-}
\newcommand{\HFm}{{\rm {HF}}^-}
\newcommand{\HFinf}{{\rm {HF}}^{\infty}}
\newcommand{\HFa}{{\widehat{{\rm {HF}}}}}

\newcommand{\DD}{{\mathcal {H}}}

\newcommand{\defin}[1]{{\bf\emph{#1}}}

\newcommand{\Hom}{\rm{Hom}}

\newcommand{\kotojel}{{\mbox{-}}}

\begin{document}
\title[Connected  Floer homology]{Connected Floer
  homology of covering involutions}

\author{Antonio Alfieri}
\address{R\'enyi Institute of Mathematics, 1053. Budapest, Re\'altanoda
utca 13-15. Hungary}
\email{alfieriantonio90@gmail.com}

\author{Sungkyung Kang}
\address{Mathematical Institute, University of Oxford, UK, and}
\address{Institute of Mathematical Sciences, the Chinese University of Hong Kong, Shatin, Hong Kong}
\email{sungkyung38@icloud.com}

\author{Andr\'{a}s I. Stipsicz} 
\address{R\'enyi Institute of Mathematics, 1053. Budapest, Re\'altanoda
utca 13-15. Hungary}
\email{stipsicz.andras@renyi.mta.hu}

\begin{abstract} 
Using the covering involution on the double branched cover of $S^3$
branched along a knot, and adapting ideas of Hendricks-Manolescu and
Hendricks-Hom-Lidman, we define new knot invariants and apply them to
deduce novel linear independence results in the smooth concordance
group of knots.
\end{abstract}
\maketitle

\section{Introduction}
\label{sec:intro}

Concordance questions of knots has been effectively studied by
4-dimensional topological methods. Indeed, for a knot $K$ in the
three-sphere $S^3$ consider the double branched cover $\Sigma (K)$ of
$S^3$ branched along $K$. If $K$ is a slice knot (i.e. bounds a
smoothly embedded disk in the 4-disk $D^4$) then $\Sigma (K)$ bounds a
four-manifold $X$ having the same rational homology as $D^4$: this $X$
can be chosen to be the double branched cover of $D^4$ along the slice
disk. The existence of such four-manifold then can be obstructed by
various methods, leading to sliceness obstructions of knots.  For
example, Donaldson's diagonalizability theorem applies in case $\Sigma
(K)$ is known to bound a negative definite four-manifold with
intersection form which does not embed into the same rank diagonal
lattice. This line of reasoning was used by Lisca in his work about
sliceness properties of 2-bridge knots, see \cite{greenejabuka,
  anapretzel, sliceribbonlisca}.  A numerical invariant (in the same
spirit) was introduced by Manolescu-Owens~\cite{ManOwens} utilizing 
the Ozsv\'ath-Szab\'o correction term of the
unique spin structure of $\Sigma (K)$.

Different knots might admit diffeomorphic double branched covers,
though; for example, if $K$ and $K'$ differ by a Conway mutation, then
$\Sigma (K)$ and $\Sigma (K')$ are diffeomorphic. This implies that if
$K$ is slice, all slice obstructions coming from the above strategy
must vanish for $K'$ as well. A long-standing problem of this type was
whether the Conway knot is slice; it admits a mutant (the
Kinoshita-Terasaka knot) which is slice, hence merely considering the
double branched cover will not provide sliceness obstruction.  (The
fact that the Conway knot is not slice has been recently proved by
Piccirillo \cite{piccirillo}, relying on four-dimensional topological methods
and results from Khovanov homology.)

The information we neglect in the above approach is that the
three-manifold $\Sigma (K)$ (viewed as the double branched cover of
$S^3$ along $K$) comes with a self-diffeomorphism $\inv$, where pairs
of points in $\Sigma (K)$ mapping to the same point of $S^3$ are
interchanged by $\inv$. In this paper we introduce modifications of
the usual Heegaard Floer homology groups of $\Sigma (K)$ which take
this $\Z /2\Z$-action into account, leading to new knot invariants.

Heegaard Floer homology associates to a closed, oriented,
smooth three-manifold a finitely generated $\Field [U]$-module $\HFm
(Y)$ (where $\Field [U]$ is the polynomial ring over the field
$\Field$ of two elements):
it is the homology of a chain complex $(\CFm (Y), \partial )$
(defined up to chain homotopy equivalence) and the homology naturally splits
according to the spin$^c$ structures of $Y$ as
\[
\HFm (Y)=\oplus _{\s \in {\rm {Spin}}^c (Y)}\HFm (Y, \s ).
\]
If $Y$ is a rational homology sphere (i.e., $b_1(Y)=0$) then $\HFm
(Y, \s )$ admits a natural $\Q$-grading, and the graded $\Field [U]$-module
$\HFm (Y, \s)$ is a diffeomorphism invariant of the spin$^c$
three-manifold $(Y, \s)$, while the local equivalence class of $(\CFm
(Y, \s), \partial )$ (for the definition of its notion see
Definition~\ref{def:LocalEquivalence}) provides an invariant of the
rational spin$^c$ homology cobordism class of $(Y, \s )$.  In this
case the local equivalence class of $\CFm (Y, \s)$ can be
characterised by a single rational number $d(Y, \s )$, the
\emph{Ozsv\'ath-Szab\'o correction term} of the spin$^c$
three-manifold $(Y, \s )$.

More recently, exploiting a symmetry built in the theory, Hendricks
and Manolescu introduced involutive Heegaard Floer homology~\cite{HM}.
The main idea of their construction was that the chain complex $\CFm
(Y)$ admits a map $\iota \colon \CFm (Y)\to \CFm (Y)$ which is (up to
homotopy) an involution, and the mapping cone of the map $\iota +\id $
provided ${\rm {HFI}}(Y)$, a module over the ring $\Field
[U,Q]/(Q^2)$. This group is interesting only for those spin$^c$
structures which originate from a spin structure, and provides a new
and rather sensitive diffeomorphism invariant of the underlying spin
three-manifold.  A further application of the above involution
appeared in the work of Hendricks, Hom and Lidman \cite{HHL}, where
connected Heegaard Floer homology $\HFm_{\rm{conn}}(Y)$, a submodule of
$\HFm (Y)$ was defined. This submodule turned out to be a homology
cobordism invariant.

Similar constructions apply for any chain complex equipped
with a (homotopy) involution.  In this paper we will define the
\emph{branched knot Floer homology} of $K$ as $\HFB (K)=H_*(\Cone
(\inv _{\#}+\id ))$, where $\inv \colon \Sigma (K)\to \Sigma (K)$ is
the covering involution, $\inv _{\#}$ is the map induced by $\inv $ on
the Heegaard Floer chain complex $\CFm (\Sigma (K),\s _0)$, 
with $\s _0$ the
unique spin$^c$ structure on $\Sigma (K)$, and $\Cone$ is the mapping
cone of the map $\inv _{\#}+\id $.  (Related constructions have been
examined in \cite{HLS}.)

\begin{thm}
\label{thm:BranchedKnotHom}
The group $\HFB (K)$, as a graded $\Field [U,Q]/(Q^2)$-module, is an
isotopy invariant of the knot $K\subset S^3$.
\end{thm}

A simple argument shows that, as an $\Field [U]$-module, $\HFB (K)$ is
isomorphic to $\Field [U]_{(\deltaover (K))} \oplus\Field
[U]_{(\deltaunder (K))}\oplus A$, where $\deltaover (K), \deltaunder
(K)\in \Q$ and $A$ is a finitely generated, graded $U$-torsion module
over $\Field [U]$.

\begin{thm}
\label{thm:TauBars} 
The rational numbers $\deltaover (K)$ and $\deltaunder(K)$ are knot
concordance invariants.
\end{thm} 

Adapting the method of \cite{HHL} for defining new homology cobordism
invariants of rational homology spheres, we define the \emph{connected
  branched Floer homology} $\HFB_{\rm {conn}}(K)$ of a knot $K\subset
S^3$ as follows.  Consider a sel-local equivalence $f_{\rm{max}}\colon
\CFm (\Sigma (K), \s_0)\to \CFm (\Sigma (K), \s _0)$ which commutes
(up to homotopy) with $\inv _{\#}$ and has maximal kernel among such
endomorphisms.  Then take $\HFB_{\rm{conn}}(K)=H_* (\Imm \,
f_{\rm{max}})$.

\begin{thm}
\label{thm:ConnKnot}
The module $\HFB_{\rm{conn}}(K)$ (up to isomorphism) is independent of
the choice of the map $f_{\rm {max}}$ with maximal kernel, and the
isomorphism class of the graded module $\HFB_{\rm{conn}} (K)$ is a
concordance invariant of the knot $K$.
\end{thm}
It follows from the construction that $\HFB_{\rm{conn}}(K)$ is an
$\Field [U]$-submodule of $\HFm (\Sigma (K), \s _0)$.  As a finitely
generated $\Field [U]$-module, $\HFB_{\rm{conn}}(K)$ is the sum of
cyclic modules, and since it is of rank one, it can be written as
\[
\HFB_{\rm{conn}}(K)=\Field [U]\oplus \HFB_{\rm{red\kotojel conn}}(K),
\]
where the second summand (the $U$-torsion submodule) is the
\emph{reduced} connected homology of $K$.

It is not hard to see that if $\HFm (\Sigma (K), \s _0)=\Field [U]$ 
holds  --- for example if $\Sigma (K)$ is an $L$-space, which is
the case if $K$ is a quasi-alternating knot ---, then $\inv _{\#}$
is chain homotopic to $\id$ implying

\begin{thm}
\label{thm:AltAndQuasiAlt}
If $K$ is concordant to an alternating (or more generally to a
quasi-alternating) knot, then $\HFB_{\rm{red\kotojel conn}}(K)=0$.
\end{thm}

Somewhat more surprisingly, the same vanishing holds for torus knots
(a phenomenon reminiscent to the behaviour of the extension of the
Upsilon-invariant of \cite{OSS4} to the Khovanov setting given by
Lewark-Lobb in \cite{LewarkLobb}):

\begin{thm}
\label{thm:TorusVanishing}
For the torus knot $T_{p,q}$ we have that $\HFB_{\rm{red\kotojel conn}}(T_{p,q})=0$.
\end{thm}

Lattice homology of N\'emethi~\cite{Nem3} (through results and
computations of Dai-Manolescu~\cite{DM} and
Hom-Karakurt-Lidman~\cite{HKL}) provides a computational scheme of the
above invariants for Montesinos (and in particular, for pretzel)
knots. The combination of some such sample calculations with the two
vanishing results above allow us to show that certain families of
pretzel knots are linearly independent from alternating and torus
knots in the smooth concordance group.  To state the results, let us introduce
the following notations: let ${\mathcal{C}}$ denote the (smooth)
concordance group of knots in $S^3$ and ${\mathcal {QA}}$
(respectively ${\mathcal {T}}$) those subgroups of ${\mathcal{C}}$
which are generated by all quasi-alternating (respectively torus)
knots. In addition, ${\mathcal {QA}}+{\mathcal {T}}$ is the subgroup
generated by alternating knots and torus knots.
The following theorem extends results from \cite{alfieri1, Wang}.

\begin{thm}
\label{thm:pretzels}
Let $K$ be the connected sum of pretzel knots of the form $P(-2,3,q)$,
with $q\geq 7$ odd. Then $K$ is not concordant to any linear
combination of alternating or torus knots, i.e.  $[K]$ is a nonzero
element of the quotient ${\mathcal {C}}/({\mathcal {T}}+{\mathcal
  {QA}})$.
\end{thm} 

Furthermore, relying on computations from \cite{DM, HKL}, we prove the following. 

\begin{thm}
\label{thm:PretzelLinIndep}
The pretzel knots $\{ P(-p,2p-1,2p+1)\mid p\ \mbox{odd}\}$ are
linearly independent in the quotient group $\mathcal
{C}/(\mathcal{T}+\mathcal{QA})$. In paricular, $\Z ^{\infty}\subset
\mathcal {C}/(\mathcal{T}+\mathcal{QA})$.
\end{thm}

The paper is organized as follows. In Section~\ref{sec:definition} we
introduce branched knot Floer homology, and in
Section~\ref{sec:ConcordanceInv} we discuss the details of the
definition of the connected  Floer homology group $\HFB_{\rm{conn}}(K)$
of a knot $K\subset S^3$.  Section~\ref{sec:Vanishing} is devoted to
the proof of the vanishing results above, while in
Section~\ref{sec:Montesinos} we give a way to compute the invariants
for Montesinos knots; finally in Section~\ref{sec:IndepResults} we derive
some independence results in the smooth concordance group.

\bigskip

{\bf Acknowledgements:} We would like to thank Andr\'as N\'emethi
for numerous highly informative discussions.  The first and the third
author acknoweldges support from the NKFIH \emph{\'Elvonal} project
KKP126683.  The second author was partially supported by European
Research Council (ERC) under the European Union Horizon 2020 research
and innovation programme (grant agreement No 674978).

\section{Definition of  branched knot Floer homology}
\label{sec:definition}
Let $\DD=(\Sigma, \boa , \bob , z)$ be a pointed Heegaard diagram
which represents a rational homology sphere $Y$, and let $J_s$ be a
generic path of almost-complex structures on the $g$-fold symmetric
product $\Sym^g(\Sigma)$ (compatible with a symplectic structure
constructed in \cite{Perutz}).  Heegaard Floer homology \cite{OS2} assigns
to the pair $(\mathcal{H},J_s)$ a finitely generated, $\Q$-graded
chain complex $\CFm(\mathcal{H}, J_s)$ over the polynomial ring
$\F[U]$, graded so that $\deg U=-2$.  This chain complex is defined as the free
$\F[U]$-module generated by the intersection point 
of the Lagrangian tori 
$\T_\alpha=\alpha_1 \times \dots \times \alpha_g$ and
$\T_\beta=\beta_1 \times \dots \times \beta_g$ in $\Sym^g(\Sigma)$,
and is equipped with the differential
\begin{equation}\label{eq:differential}
\partial \x = \sum_{\y \in \T_\alpha \cap \T_\beta} \sum_{\{\phi \in
  \pi_2(\x,\y) | \mu(\phi)=1 \}} 
\# \left( \mathfrak{\mathcal{M}(\phi)}/{\R}\right) 
U^{n_z(\phi)} \cdot \y
\end{equation}
where $\# ( \mathcal{M}(\phi)/\R ) $ is the (mod 2)
number of points in the unparametrized moduli space
$\mathcal{M}(\phi)/\R$ of $J_s$-holomorphic strips with index
$\mu(\phi)=1$ representing the homotopy class $\phi \in \pi _2 (\x,
\y)$ and $n_z(\phi )$ is the intersection number
of $\phi $ with the divisor $V_z=\{ z\}\times \Sym ^{g-1}(\Sigma )$. 
For more details about Heegaard Floer homology see 
\cite{OS4, OS2, OS6A, OS6B}.

For a knot $K\subset S^3$, let $\Sigma (K)$ denote the double branched
cover of $S^3$ branched along $K$.  The three-manifold $\Sigma (K)$
comes with a natural map $\inv\colon \Sigma (K) \to \Sigma (K)$
(called the \emph{covering involution}) which interchanges points with
equal image under the branched covering map $\pi \colon \Sigma (K)\to
\Sigma(K)/\inv\simeq S^3$.  The fixed point set ${\rm{Fix}}(\inv
)={\widetilde {K}}$ maps homeomorphically to $K$ under $\pi$.  As the
notation suggests (since $H_1(S^3\setminus K; \Z )\cong \Z$), the
branched cover $\Sigma (K)$ in this case is determined by the branch
locus $K\subset S^3$.

Pulling back the Heegaard surface, as well as the $\boa$- and the
$\bob$-curves of a doubly-pointed Heegaard diagram $D=(\Sigma, \boa,
\bob, w_1, w_2)$ representing $K\subset S^3$ we get a pointed Heegaard
diagram
\[
\DD_D=\left(\widetilde{\Sigma}=\pi^{-1}(\Sigma),
\widetilde{\boa}=\pi^{-1}(\boa), \widetilde{\bob}=\pi^{-1}(\bob),
z=\pi^{-1}(w_1)\right)
\] 
of the double branched cover $\Sigma(K)$.  The covering projection
$\pi \colon \Sigma (K)\to S^3$ restricts to a double branched cover of
Riemann surfaces $\pi|_{\widetilde{\Sigma}}\colon \widetilde{\Sigma}
\to \Sigma$ with branch set $\{w_1, w_2\}$.
The restriction $\inv|_{\widetilde{\Sigma}}\colon
\widetilde{\Sigma} \to \widetilde{\Sigma}$ of the covering involution
$\inv$ represents the covering
involution of $\pi|_{\widetilde{\Sigma}}\colon \widetilde{\Sigma} \to
\Sigma$, and $\inv|_{\widetilde{\Sigma}}\colon \widetilde{\Sigma} \to
\widetilde{\Sigma}$ induces a self-diffeomorphism of the symmetric
product
\begin{equation}\label{eq:InvOnSym}
\oinv\colon \Sym^g(\widetilde{\Sigma})\to\Sym^g(\widetilde{\Sigma}) \ ,
\end{equation}
leaving $\T_{\widetilde{\alpha}}$ and $\T_{\widetilde{\beta}}$, as
well as the divisor $V_z=\{ z\}\times \Sym^{g-1}(\widetilde{\Sigma})$
invariant.

Pick a generic path of almost-complex structures $J_s\in
\mathcal{J}(\Sym^g(\widetilde{\Sigma}))$ (satisfying the usual
compatibility conditions with the chosen symplectic form on
the symmetric product) and consider the Heegaard Floer chain
complex $\CFm(\DD_D,J_s)$ associated to $(\DD_D, J_s)$. Recall that
there is a direct sum decomposition of $\CFm (\DD_D, J_s)$ indexed by
$\spinc$ structures:
\[
\CFm(\DD_D,J_s)= \bigoplus_{\s \in \Spinc(\Sigma(K))} \CFm(\DD_D,J_s; \s) \ . 
\] 

The first (singular) homology group of the double branched cover
$\Sigma(K)$ can be presented by $\theta+\theta^t$, where $\theta$ is a
Seifert matrix for $K$. Thus, $|H_1(\Sigma(K), \Z)|=\det
(\theta+\theta')=\det (K)$, which is an odd number. In particular,
$\Sigma(K)$ has a unique spin structure $\s_0$.  We will focus
on $\CFm(\mathcal{H}_D, J_s; \s_0)$, the summand of
$\CFm(\DD_D, J_s)$ associated to $\s_0$.

Note that if the path of almost-complex structures $J_s\in
\mathcal{J}(\Sym^g(\widetilde{\Sigma}))$ is chosen generically,
transversality is achieved for both $J_s$ and the push-forward
$\oinv_*J_s$ (where $\oinv$ is given in
Equation~\eqref{eq:InvOnSym}). For such a choice of almost-complex
structures we have well-defined Heegaard Floer chain complexes
$\CFm(\mathcal{H}_D, J_s)$ and $\CFm (\DD_D, \oinv_*J_s)$, and we can
consider the map
\begin{equation}\label{eq:eta}
\eta\colon \CFm(\mathcal{H}_D, J_s)\to \CFm(\mathcal{H}_D, \oinv_*J_s)
\end{equation}
sending a generator $\x=x_1+ \dots +x_g \in \T_{\widetilde{\alpha}}
\cap \T_{\widetilde{\beta}}\subset \Sym ^g ({\widetilde {\Sigma }})$
to $\oinv(\x)=\inv (x_1)+ \dots +\inv (x_g)$.

\begin{lem} 
The map $\eta$ is an isomorphism of chain complexes. Furthermore,
$\eta$ maps the summand $\CFm(\mathcal{H}_D, J_s; \s_0)$ 
of the spin structure $\s_0$ into itself.
\end{lem}
\begin{proof} 
It is obviously an isomorphisms of free $\F[U]$-modules; indeed,
$\eta^2=\id$ since $\tau$ is an involution.  To see that $\eta$
commutes with the differential, notice that $u \mapsto \inv \circ u$
provides a diffeomorphism between the moduli space of
$J_s$-holomorphic representatives of a homotopy class $\phi \in
\pi_2(\x,\y)$ and the moduli space of $\oinv_*J_s$-holomorphic
representatives of $\inv\circ \phi\in \pi_2(\inv(\x), \inv(\y))$).

To show that $\eta$ preserves the spin structure we argue as
follows. According to \cite[Section 2.6]{OS2} the choice of a
basepoint $z$ of $\DD_D$ determines a map $\s_z\colon
\T_{\widetilde{\alpha}} \cap \T_{\widetilde{\beta}} \to
\Spinc(\Sigma(K))$ and
\[ 
\CFm(\DD_D, J_s; \s)= \bigoplus_{\s_z(\x)=\s} \F[U] \cdot \x \ .
\]
It follows from the definition of $\s_z$ that
$\tau_*(\s_z(\x))=\s_z(\tau(\x))$. Thus if $\s_z(\x)=\s_0$, we have
that 
\[
\overline{\s_z(\tau(\x))}=\overline{\tau_*(\s_z(\x))}=
\overline{\tau_*(\s_0)}=\tau_*(\overline{\s_0})=\tau_*(\s_0)=\s_z(\tau(\x))
\]
proving that $\s_z(\tau(\x))$ is a self-conjugate $\spinc$ structure,
{\rm{i.e.}} spin. The claim now follows from the fact that
$\Sigma(K)$ has a unqiue spin structure.
\end{proof}

We define 
$\inv_\#\colon \CFm(\mathcal{H}_D, J_s; \s_0)\to \CFm(\mathcal{H}_D, J_s; \s_0)$
as the map $\eta\colon \CFm(\mathcal{H}_D, J_s; \s_0)\to
\CFm(\mathcal{H}_D, \tau_*J_s; \s_0)$ followed by the continuation map
\[
\Phi_{J_{s,t}}^-\colon\CFm(\mathcal{H}_D, \inv_*J_s; \s_0) \to
\CFm(\mathcal{H}_D, J_s; \s_0)
\] 
from \cite[Section 6]{OS2}, induced by a generic two-parameter family
$J_{s,t}$ of almost-complex structures interpolating between $J_s$
and $\inv_*J_s$:
\[
\inv _\# = \Phi _{J_{s,t}}\circ \eta .
\]

\begin{lem} 
$\inv_\#^2 \simeq \id$, where $\simeq$ denotes chain homotopy
  equivalence.
\end{lem}
\begin{proof} 
Consider
\[
\Phi_{J_{s,t}}^-(\x)=\sum_{\y \in \T_{\widetilde{\alpha}} \cap
  \T_{\widetilde{\beta}}} \sum_{ \{\phi\in \pi_2(\x,\y) \ |
  \ \mu(\phi)=0\} } \# \left( \mathcal{M}_{J_{s,t}}(\phi)\right)
\ U^{n_z(\phi)} \cdot \y
\] 
where $\mathcal{M}_{J_{s,t}}(\phi)$ denotes
the moduli spaces of $J_{s,t}$-holomorphic strips.

Given $\x \in \T_{\widetilde{\alpha}} \cap \T_{\widetilde{\beta}}$ one
computes
\begin{align*}
\eta\circ \Phi_{J_{s,t}}^-(\x)&=\sum_{\y \in \T_{\widetilde{\alpha}}
  \cap \T_{\widetilde{\beta}}} \sum_{ \{\phi\in \pi_2(\x,\y) \ |
  \ \mu(\phi)=0\} } \#\left(\mathcal{M}_{J_{s,t}}(\phi)\right)
\ U^{n_z(\phi)} \cdot \tau(\y)\\ &=\sum_{\y \in
  \T_{\widetilde{\alpha}} \cap \T_{\widetilde{\beta}}} \sum_{
  \{\phi\in \pi_2(\x,\tau(\y)) \ | \ \mu(\phi)=0\} } \# \left(
\mathcal{M}_{J_{s,t}}(\phi)\right) \ U^{n_z(\phi)} \cdot
\y\\ &=\sum_{\y \in \T_{\widetilde{\alpha}} \cap
  \T_{\widetilde{\beta}}} \sum_{ \{\phi\in \pi_2(\tau(\x),\y) \ |
  \ \mu(\phi)=0\} } \# \left( \mathcal{M}_{J_{s,t}}(\inv \circ
\phi)\right) \ U^{n_z(\tau \circ \phi)} \cdot \y\\ &=\sum_{\y \in
  \T_{\widetilde{\alpha}} \cap \T_{\widetilde{\beta}}} \sum_{
  \{\phi\in \pi_2(\tau(\x),\y) \ | \ \mu(\phi)=0\} } \# \left(
\mathcal{M}_{J_{s,t}}(\phi) \right) \ U^{n_z(\phi)} \cdot
\y\\ &=\Phi_{\tau_*J_{s,t}}^-(\tau(\x))\ ,
\end{align*}
hence the identity $\eta\circ \Phi_{H}^-=\Phi_{\tau_*H}^-\circ \eta$
follows. Thus,
\[
\inv _\#^2= \Phi_{J_{s,t}}^- \circ \eta \circ \Phi_{J_{s,t}}^- \circ
\eta= \Phi_{J_{s,t}}^- \circ \Phi_{\tau_*J_{s,t}}^-\circ \eta^2=
\Phi_{J_{s,t}}^- \circ \Phi_{\tau_*J_{s,t}}^- \ ,
\] 
where the last identity holds because $\eta^2=\id$ (a consequence of
the fact that $\tau\colon \Sigma(K) \to \Sigma(K)$ is an involution).

By concatenating $J_{s,t}$ and $\tau_*J_{s,t}$ we obtain a
one-parameter family of paths of almost-complex structures describing a closed
loop based at the path $J_{s}$. Since the space of almost complex
structures (compatible with the fixed symplectic structure) is
contractible, we can find a three-parameter family of almost complex
structures $J_{s,t,x}$ interpolating between the juxtaposition of
$J_{s,t}$ and $\tau_*J_{s,t}$, and $J_{s,t,1}\equiv J_{s}$. As pointed
out in \cite[Section 6]{OS2}, a generic choice of $J_{s,t,x}$ produces
smooth moduli spaces
\[\ \ \ \ \ \ \  \ \ \ \ \ \  \mathcal{M}_{J_{s,t,x}}(\phi)= \bigcup_{c \in [0,1]} \mathcal{M}_{J_{s,t,c}}(\phi) \ \ \ \ \ \ \  \ \phi \in \pi_2(\x,\y) \]  
of dimension $\mu(\phi)+1$. These can be used to produce a chain homotopy equivalence 
\[
H^-_{J_{s,t,x}}(\x) =\sum_{\y \in \T_{\widetilde{\alpha}} \cap
  \T_{\widetilde{\beta}}} \sum_{ \{\phi\in \pi_2(\x,\y) \ |
  \ \mu(\phi)=-1\} } \# \left(\mathcal{M}_{J_{s,t,x}}(\phi)
\right)\ U^{n_z(\phi)} \cdot \y
\]
between $\Phi_{J_{s,t}}^- \circ \Phi_{\tau_*J_{s,t}}^-$ and $\id$, concluding
the argument.
\end{proof}

In summary, for a knot $K\subset S^3$ there is a homotopy involution
$\inv_\#\colon \CFm(\Sigma(K),\s_0) \to \CFm(\Sigma(K),\s_0)$
associated to the covering involution $\inv\colon \Sigma(K) \to
\Sigma(K)$.  In order to derive knot invariants from the pair $(\CFm
(\DD, \s_0), \inv _\#)$, we follow ideas from \cite{HM} and form the
mapping cone of $\inv _\# + \id\colon \CFm(\Sigma(K),\s_0)\to
\CFm(\Sigma(K),\s_0)$, written equivalently as
\[
\CFB (K)= \left( \CFm (\Sigma(K),\s_0)[-1] \otimes \F[Q]/ (Q^2),
\ \partial_{\rm{cone}}= \partial + Q \cdot (\inv_\# +\id) \right), 
\]   
where $\deg Q=-1$. Recall that $\CFm (\Sigma (K), \s _0)$ admits an
absolute $\Q$-grading, and $\inv _\#$ preserves this grading, hence
$\CFB (K)$ also admits an absolute $\Q$-grading.  Taking homology we
get the group $\HFB (K)=H_*(\CFB(K))$, which is now a module over the
ring $\Field [U,Q]/(Q^2)$. We call $\HFB (K)$ the \defin{branched
  Heegaard Floer homology} of the knot $K\subset S^3$.  Now we are
ready to turn to the proof of the first statement announced in
Section~\ref{sec:intro}:

\begin{proof}[Proof of Theorem~\ref{thm:BranchedKnotHom}] 
The proof is similar to the one of \cite[Proposition 2.8]{HM}.
Independence from the chosen path of almost-complex structures is
standard Floer theory. For independence from the chosen doubly pointed
Heegaard diagram of $K$, we argue as follows: A doubly pointed
Heegaard diagram $D=(\Sigma, \boa, \bob, w_1, w_2)$ representing the
knot $K \subset S^3$ can be connected to any other doubly pointed
diagram $D '=(\Sigma', \boa', \bob', w_1 ', w_2 ')$ of $K$ by a
sequence of isotopies and handleslides of the $\boa$-curves (or
$\bob$-curves) supported in the complement of the two basepoints, and
by stabilizations (i.e., forming the connected sum of $\Sigma$ with a
torus $T^2$ equipped with a new pair of curves $\alpha_{g+1}$ and
$\beta_{g+1}$ which meet transversally in a single point).
A sequence of these moves lifts to a sequence of pointed Heegaard
moves of the pull-back diagrams $\DD_D$ and $\DD_{D'}$ with underlying
three-manifold the double branched cover $\Sigma(K)$. According to
\cite{JT} the choice of such a sequence of Heegaard moves yields a
natural chain homotopy equivalence $\psi\colon \CFm (\DD_D) \to
\CFm (\DD_{D'})$ which fits into the diagram
\begin{equation} \label{diagramFirst}
\xymatrix{\CFm (\DD_{D}) \ar[r]^{\tau_\# } \ar[d]^{\psi} & \CFm (\DD_{D})  \ar[d]^{\psi}\\ 
\CFm (\DD_{D'}) \ar[r]^{\tau_\#' }  &\CFm (\DD_{D'})\\} 
\end{equation}
that commutes up to chain homotopy. Let $\Gamma\colon \CFm (\DD_{D})
\to \CFm (\DD_{D'})$ be a map realizing the chain homotopy
equivalence. Then the map $f\colon {\rm{Cone}}(\tau_\#+\id) \to
{\rm{Cone}}(\tau_\#'+\id )$ defined by $f= \psi+ Q \cdot (\psi +\Gamma)$
is a quasi-isomorphism. Indeed, $f$ is a filtered map with respect to
the two step filtration of the mapping cones, and since it induces an
isomorphism on the associated graded objects, it is a
quasi-isomorphism.
\end{proof}

Notice that the mapping cone exact sequence associated to $\CFB
(K)=\Cone (\inv _\# +\id)$ reads as an exact triangle
\begin{equation}\label{eq:exact}
\begin{tikzcd}[column sep=small]
 \HFm (\Sigma(K),\s_0) \arrow{rr}{\inv_* +\id} &     & \HFm (\Sigma(K), \s_0) \arrow{ld}{j_*} \\
&\HFB (K)  \arrow{lu}{p_*}   & 
\end{tikzcd}
\end{equation}
in which $j_*$ preserves the grading, and $p_*$ drops it by one.  In
particular, if $\inv _*=\id$, the horizontal map in the above triangle
is zero, and in that case $\HFB (K)$ is the sum of two copies of $\HFm
(\Sigma (K), \s _0)$ (with the grading on one copy shifted by one).

A close inspection of the exact triangle above reveals that, as
$\Field [U]$-modules, we have
\[
\HFB (K)= \F[U]_{(\overline{\delta})} \oplus \F[U]_{(\underline{\delta}+1)} \oplus \left( \F[U]{\rm{-torsion}}\right) \ . 
\]
We set $\underline{\delta}(K)=\underline{\delta}$ and
$\overline{\delta}(K)=\overline{\delta}$, which (by the above
discussion) are knot invariants. Notice that $\underline{\delta}(K),
\overline{\delta}(K)\in \Q$, $\underline{\delta}(K) \equiv \delta (K)
\equiv \overline{\delta}(K) \ \rm{ mod} \ 2$, and
$\underline{\delta}(K) \leq \delta(K) \leq \overline{\delta}(K)$,
where $\delta (K)$ is the Ozsv\'ath-Szab\'o correction term of
$(\Sigma (K), \s _0)$, hence $\delta (K)$ is half the Manolescu-Owens
invariant of $K$ introduced in \cite{ManOwens}.

\section{Concordance invariants from $(\CFm (\Sigma (K), \s _0), \inv _\#$)}
\label{sec:ConcordanceInv}

Adapting ideas from \cite{HHL}, the chain complex $\CFm (\Sigma (K), \s _0)$,
equipped with $\inv _\#$, provides concordance invariants of the knot  $K$ 
as follows. Recall \cite[Definition~2.5]{HHL} regarding $\iota$-complexes:

\begin{defi} 
\label{def:iotaComplex}
An \defin{$\iota$-complex} $(C, \iota)$ is a finitely generated, free,
$\Q$-graded chain complex $C$ over $\F[U]$ together with a chain map
$\iota \colon C \to C$ where $C$ is supported in degree $\tau + \Z$
for some $\tau \in \Q$
(multiplication by $U$ drops the $\Q$-grading by two), the
homology of the localization $U^{-1}H_*(C)=H_*(C \otimes_{\F[U]}
\F[U,U^{-1}])$ is isomorphic to $\F[U,U^{-1}]$ via an isomorphism
preserving the relative $\Z$-grading, and $\iota $ is a grading preserving,
$U$-equivariant chain map  which is a homotopy involution.
\end{defi}

We will consider $\iota$-complexes up to \textit{local equivalence}
(see \cite[Definitions~2.6 and 2.7]{HHL}).
\begin{defi}
\label{def:LocalEquivalence}
A \defin{local equivalence} $f\colon C \to C'$ of two
$\iota$-complexes $(C, \iota)$ and $(C', \iota')$ is a grading
preserving, $U$-equivariant chain map $f\colon C \to C'$ such that
\begin{itemize}
\item $\iota' \circ  f \simeq f \circ \iota$, i.e. the two 
compositions are  chain homotopy equivalent,
\item $f$ induces an isomorphisms $f_\text{loc}$ on  the localization $U^{-1}H_*(C)$.
\end{itemize}
\end{defi}

\begin{defi}
Two $\iota$-complexes $(C, \iota)$ and $(C', \iota')$ are
\defin{locally equivalent} if there exist local equivalences $f\colon C
\to C'$ and $g\colon C' \to C$. If in addition we have $f \circ
g\simeq \id$ and $g \circ f \simeq \id$, then $(C,
\iota)$ and $(C', \iota')$ are \defin{chain homotopy equivalent}
$\iota$-complexes. 
\end{defi}

Given an $\iota$-complex $(C, \iota)$ we can look at the set
$\text{End}_\text{loc}(C, \iota)$ of its self-local equivalences
$f\colon C \to C$. This can be partially ordered by defining $f
\preceq g$ if and only if ${\Ker}\, f \subset {\Ker}\, g$. We say that
$f\in {\rm{End}}_{\rm{loc}}(C)$ is a \defin{maximal self-local
  equivalence} if it is maximal with respect to this ordering.  By
Zorn's lemma maximal self-local equivalences always exist. The
following lemma summarises the results of \cite[Section 3]{HHL}.

\begin{lem}
\label{lem:algebra}
Let $(C, \iota)$ be an $\iota$-complex. Then 
\begin{enumerate}
\item \label{due} if $f \in {\rm{End}}_{\rm{loc}}(C, \iota)$ is a maximal
  self-local equivalence, then $\iota$ restricts to a homotopy involution $\iota^{\Imm \, f}$ of $\Imm \, f$. Furthermore, $(\Imm \, f, \iota^{\Imm \, f})$ is locally equivalent to $(C, \iota)$;
\item \label{tre} if $f,h \in {\rm{End}}_{\rm{loc}}(C, \iota)$ are two
  maximal self-local equivalences, then there is a chain homotopy
  equivalence $(\Imm \, f , \iota^{\Imm \, f}) \simeq (\Imm \, h, \iota^{\Imm \, h})$ of $\iota$-complexes;
\item \label{quattro} if $(C', \iota')$ is an $\iota$-complex locally
  equivalent to $(C, \iota)$ and $f \in {\rm{End}}_{\rm{loc}}(C,\iota)$, and $h
  \in {\rm{End}}_{\rm{loc}}(C', \iota')$ are self-local equivalences then
  there is a chain homotopy equivalence $(\Imm \, f, \iota^{\Imm \, f}) \simeq (\Imm \, h, \iota^{\Imm \, h})$ of
  $\iota$-complexes. \qed
\end{enumerate}
\end{lem}
Since
\[
H_*(\CFm (\Sigma (K), \s_0) \otimes_{\F[U]} \F[U,U^{-1}])=\HFinf
(\Sigma(K), \s_0)= \F[U, U^{-1}]\ ,
\]
the pair $(\CFm (\Sigma(K), \s_0), \tau_\#)$ of the Heegaard Floer
chain complex of the double branched cover of a knot $K\subset S^3$
(equipped with the homotopy involution $\inv _\#$ induced by the
covering involution) is an $\iota$-complex associated to $K$.  Given a
maximal self-local equivalence $f_{\rm{max}}\colon \CFm (\Sigma(K),
\s_0) \to \CFm (\Sigma(K), \s_0)$ we define $\HFB_{\rm{conn}}(K)$, the
\defin{connected  Floer homology} of the knot $K\subset S^3$ as
$H_*({\Imm} \, f_{\rm{max}})$.  As an application of Lemma
\ref{lem:algebra}, it is then easy to see that the resulting group is a
knot invariant:
 
\begin{thm}\label{thm:ConnKnotC}
The chain homotopy type of $\Imm \, f_{\max}$ is independent of the
choice of the maximal self-local equivalence $f_{\max} \in
{\rm{End}}_{\rm{loc}}(\CFm (\Sigma(K), \s_0), \tau_\#)$. \qed
\end{thm} 

We now turn to the proof of concordance invariance of the groups
$\HFB_{\rm{conn}}(K)$. The following naturality statement will
be needed in the proof.

\begin{lem}(Ozsv\' ath \& Szab\' o, Zemke, \cite{OS24,Z1})
\label{lem:naturalitylemma} 
Let $Y$ and $Y'$ be two three-manifolds equipped with self-diffeomorphisms
$\tau\colon Y \to Y$ and $\tau'\colon Y' \to Y'$. Suppose that
$W\colon Y \to Y'$ is a cobordism and that there exists a
self-diffeomorphism $T\colon W \to W$ restricting to $\tau$ and
$\tau'$ on the two ends of $W$. Then
\[\tau_\#' \circ F_{W, \mathfrak{t}}^-  =F_{W, T_* \mathfrak{t}}^- \circ \tau_\# \]
for every $\mathfrak{t} \in \Spinc(W) $. \qed
\end{lem} 

\begin{proof}[Proof of Theorem~\ref{thm:ConnKnot}]
Suppose that $K'\subset S^3$ is concordant to $K$, {\rm{i.e.}}  there
exists a smoothly embedded annulus $C\subset S^3\times [0,1]$ with
$\partial C= C \cap S^3\times [0,1]= K \times \{ 1\} \cup K'\times \{
0\}$. By taking the double branched cover $\Sigma(C)$ of $S^3\times
[0,1]$ branched along $C$ we get a smooth rational homology cobordism
from $\Sigma(K)$ to $\Sigma(K')$. By  adapting
\cite[Lemma~2.1]{grigsby2008knot} for $n=2$, we get that the
four-manifold $\Sigma(C)$ comes with a distinguished spin structure
$\mathfrak{t}$ restricting to the canonical spin structure on the two
ends.  In addition, this spin structure is
invariant under the covering involution of the double branched cover
$\Sigma (C)$.  Let $F_{C}^-\colon \CFm(\Sigma(K), \s_0) \to
\CFm(\Sigma(K'), \s _0)$ denote the cobordism map induced by $(\Sigma(C),
\mathfrak{t})$.

Since $\Sigma (C)$ is a rational homology cobordism, it follows that
\[
F_C^-\colon \CFm (\Sigma(K), \s_0)\to \CFm (\Sigma(K'), \s_0)
\] 
and 
\[
F_{-C}^-\colon \CFm (\Sigma(K'), \s_0)\to \CFm (\Sigma(K), \s_0) \ ,
\] 
are local equivalences. (Recall that according to \cite{OS24} a
rational homology cobordism induces an isomorphism on
$\HFinf=U^{-1}\HFm $.) The fact that $F_C^-$ and $F_{-C}^-$ homotopy
commute with the homotopy $\Z/2\Z$-actions follows from
Lemma~\ref{lem:naturalitylemma} and the fact that (in the notations of
that lemma) we have $T_*{\mathfrak{t}}={\mathfrak{t}}$. Then
Lemma~\ref{lem:algebra} concludes the argument.
\end{proof}

\begin{proof}[Proof of Theorem~\ref{thm:TauBars}]
Let $f \in {\rm{End}}_{\rm{loc}}(\CFm (\Sigma(K), \s_0),\tau_\#)$
be a maximal self-local equivalence.  As a consequence of
Lemma~\ref{lem:algebra}, the chain homotopy type of the mapping cone of the
restriction 
$\tau_\#^{\Imm \, f}+\id \colon \Imm \, f \to \Imm \, f$ is a
concordance invariant of $K$. On the other hand,
\[
H_*({\rm{Cone}}(\tau_\#^{\Imm \, f}+\id))= \F[U]_{\deltaunder (K)}
\oplus \F[U]_{\deltaover (K)+1} \oplus(\F[U]{\rm{-torsion}}),
\]
implying the claim.
\end{proof} 

Since $\HFm (\Sigma (K), \s _0)$ is of rank one (as an $\Field [U]$-module),
and $f_{{\rm {max}}}$ is a local equivalence, it follows that 
$\HFB_{\rm{conn}}(K)
\subset \HFm (\Sigma (K), \s _0)$ is also of rank one.
The $U$-torsion submodule of $\HFB_{\rm{conn}}(K)$
is the \defin{reduced connected Floer homology} of $K$, and it will be
denoted by  $\HFB_{\rm{red\kotojel conn}}(K)$.

The following simple adaptation of \cite[Proposition 4.6]{HHL} allows
us to prove triviality of $\HFB_{\rm{red\kotojel conn}}(K) $.

\begin{prop}
\label{prop:Condition}
  $\HFB_{\rm{red\kotojel conn}}(K)=0$ 
if and only if $\deltaover
  (K)=\deltaunder(K)=\delta(K)$. \qed
\end{prop}
Given a knot $K\subset S^3$  we denote by $-K$ its mirror image.

\begin{lem}
\label{lem:mirror}
For a knot $K\subset S^3$ we have that $\deltaunder(K)=-\deltaover(-K)$
\end{lem}
\begin{proof} 
The double branched cover of $-K$ is $-\Sigma(K)$. The
argument of \cite[Proposition 5.2]{HM} provides the
claimed identity.
\end{proof}

\begin{lem}
\label{lem:Sums}
If  $K=K_1\# K_2$ for two knots $K_1$ and $K_2\subset S^3$ then 
\begin{equation}\label{ineq}
\deltaunder(K_1)+ \deltaunder(K_2) \leq \deltaunder(K) \leq \deltaover(K) \leq \deltaover(K_1) + \deltaover(K_2).
\end{equation}
\end{lem}
\begin{proof} 
Suppose that $D _i$ is a doubly pointed Heegaard diagram for
$K_i\subset S^3$ ($i=1,2$). Then a doubly pointed Heegaard diagram $D$
can be constructed for $K$ by taking the connected sums of $D_i$
(along $w_2^1$ in $D_1$ and $w_2^2$ in $D _2$).  This construction
shows that $\CFm (\DD _D)$ is the tensor product of $\CFm (\DD
_{D_1})$ and $\CFm (\DD _{D_2})$. It then obviously follows that the
map $\eta _D$ of Equation~\eqref{eq:eta} for $\DD _D$ is the tensor
product of the similar maps $\eta _{D_1}$ and $\eta _{D_2}$ for $\DD
_{D_1}$ and $\DD _{D_2}$. This implies that $\tau_{D}$ and $\tau_{D_1}
\otimes \tau_{D_2}$ are chain homotopic maps, from which
\cite[Proposition 1.3]{HMZ} implies the result.
\end{proof}

\section{Vanishing results}
\label{sec:Vanishing} 
In some cases $\HFB (K)$ and $\HFB_{\rm{conn}}(K)$ can be easily
determined. As customary in Heegaard Floer theory, these invariants do
not capture any new information for alternating (or, more generally
quasi-alternating) knots. It is a more surprising (and as we
will see, very convenient) feature of $\HFB_{\rm{conn}}$ that it is
rather trivial for torus knots as well. In this section we show some
vanishing results about the group $\HFB_{\rm{conn}}(K)$, while the next section
provides methods to determine our invariants for Montesinos (and more generally
for arborescent) knots. We start with a simple motivating example.

\subsection{An example}
\label{ssec:example}
Consider the Brieskorn sphere $\Sigma (2,3,7)$; it can be given as
$(-1)$-surgery on the right-handed trefoil knot $T_{2,3}$. It is an
integral homology sphere with Heegaard Floer homology $\HFa (\Sigma
(2,3,7))=\Field ^2_{(0)}\oplus \Field _{(-1)}$ and $\HFm (\Sigma
(2,3,7))=\Field [U]_{(-2)}\oplus \Field _{(-2)}$ in its unique
spin$^c$ (hence spin) structure,
see~\cite[Equation~(25)]{ozsvath2003absolutely}.

This three-manifold can be presented as the double branched cover of
$S^3$ either along the torus knot $T_{3,7}$, or along the pretzel knot
$P(2, -3, -7)$. The two presentations provide
two involutions on $\Sigma (2,3,7)$, which potentially provide two
different maps on the Heegaard Floer chain complex. Indeed, let $\phi _1$
denote the involution $\Sigma (2,3,7)$ admits as double branched cover
along $T_{3,7}$ and let $\phi _2$ denote the involution it gets as
double branched cover along $P(2,-3,-7)$. Through direct calculation,
the actions of these maps on Heegaard Floer homology has been
identified in \cite[Propositions~6.26~and~6.27]{HLS}.

\begin{thm}[\cite{HLS}]
The map $(\phi _1)_*$ induces the identity map on $\HFa (\Sigma
(2,3,7))$ (and hence on $\HFm (\Sigma (2,3,7))$, while the map $(\phi
_2)_*$ is different from the identity on $\HFa (\Sigma (2,3,7))$. \qed
\end{thm}

This allows us to compute the invariants $\HFB $ and $\HFB_{{\rm
    {conn}}}$ for $T_{3,7}$ and $P(2,-3,-7)$, showing that
\begin{itemize}
\item $\HFB (T_{3,7})=\Field [U]_{(-2)}\oplus \Field [U] _{(-3)}\oplus
  \Field _{(-2)}\oplus \Field _{(-3)}$, or, as an $\Field
         [U,Q]/(Q^2)$-module (and ignoring gradings) $\HFB (T_{3,7})=(\Field
         [U,Q]/(Q^2))\oplus \Field ^2$,
\item $\HFB _{{\rm {conn}}}(T_{3,7})=\Field [U]_{(-2)}$, hence
  $\HFB_{\rm{red\kotojel conn}}(T_{3,7})=0$; and
\item $\HFB (P(2,-3,-7))=(\Field [U,Q]/(Q^2)) \oplus \Field $, 
\item $\HFB _{{\rm {conn}}} (P(2,-3,-7))=\Field [U]_{(-2)}\oplus
  \Field _{(-2)}$, hence $\HFB_{\rm{red\kotojel
      conn}}(P(2,-3,-7))=\Field _{(-2)}\neq 0$.
\end{itemize}

These calculations generalize to show that any torus knot has trivial
reduced connected Floer homology $\HFB_{\rm{red\kotojel
    conn}}$, while for pretzel knots there is a combinatorial method
to determine this quantity. In particular, 
the above results will be reproved in Subsection~\ref{subsec:Torus}
and in Section~\ref{sec:IndepResults}.

\subsection{Quasi-alternating knots}

\begin{proof}[Proof of Theoem~\ref{thm:AltAndQuasiAlt}]
If $K$ is an alternating (or more generally, quasi-alternating) knot,
then the double branched cover $\Sigma (K)$ is an $L$-space, and hence
$\HFm (\Sigma (K), \s _0)=\Field [U]$; in particular the homology is
only in even degrees.  Results of \cite{DM} imply that $\inv _\#$ is
determined (up to homotopy) by its action on homology, which (for a
grading preserving map) for $\Field [U]$ must be equal to the
identity.  Therefore $\inv _{\#}$ is homotopic to the identity, and so
$\inv _{\#}+\id =0$, hence the exact triangle of
Equation~\eqref{eq:exact} determines $\HFB (K)$ as the sum of two
copies of $\HFm (\Sigma (K), \s _0)$ (one with shifted grading).
Furthermore, since the homotopy commuting assumption of a self-local
equivalence in this case is vacuous, we get that
$\HFB_{\rm{conn}}(K)=\Field [U] (=\HFm (\Sigma (K), \s _0))$,
$\HFB_{{\rm{red\kotojel conn}}}(K)=0$, and the only invariant we get
from this picture is the $d$-invariant of $(\Sigma (K), \s _0)$, which
is (half of) the Manolescu-Owens invariant of the knot $K$ from
\cite{ManOwens}.
\end{proof}

\subsection{Torus knots}
\label{subsec:Torus}
Next we turn to the discussion of Theorem~\ref{thm:TorusVanishing}.
The proof of this result is significantly easier 
when $pq$ is odd.

\begin{prop}
\label{prop:TorusOdd}
Suppose that $pq$ is odd. Then the covering involution $\inv$
on the double branched cover $\Sigma (T_{p,q})$ is isotopic to $\id$. 
\end{prop}
\begin{proof}
The double branched cover of the torus knot $T_{p,q}$ is diffeomorphic
to the link of the complex surface singularity given by the  equation
$z^2=x^p+y^q$, which is the  Brieskorn sphere $\Sigma(2,p,q)$.  The 
covering involution $\tau\colon \Sigma(2,p,q) \to \Sigma(2,p,q)$ of
$T_{p,q}$ acts as $(z,x,y) \mapsto (-z,x,y)$.  

Fix $t\in S^1$ and consider the  diffeomorphism
\[
(z,x,y)\mapsto (t^{pq}z, t^{2q}x, t^{2p}y).
\]
Clearly we get an $S^1$-family of diffeomorphisms, where $t=1$
gives $\id$, while $t=-1$ (under the condition  $pq$ odd) gives 
$\inv$, concluding the proof of the proposition.
\end{proof} 

Next we turn to  the  case when exactly one of $p$ and $q$ is
even. The proof in this case requires some preparation from lattice homology.
(For a more thorough introduction to this subject see 
\cite{Nem3, OSSKnotLattice}.)

\subsection{Lattice homology}
\label{ssec:lattice}
In determining the $\iota$-complex associated to the double branched
cover of a torus knot $T_{p,q}$, the concept of lattice homology will
be extremally useful. This theory was motivated by computational
results of Ozsv\'ath and Szab\'o (for manifolds given by negative
definite plumbing trees of at most one 'bad' vertex) in \cite{OS20}
and extended by N\'emethi~\cite{Nem3} to any negative definite
plumbing trees. The isomorphism of lattice homology with Heegaard
Floer homology was established for almost-rational graphs by N\'emethi
in \cite{Nem3}, which was extended in \cite{OSSspectral} to graphs
with at most two 'bad'vertices. Below we recall the basic concepts and
results of this theory.

A \defin{graded root} is a pair $(R,w)$ where
\begin{itemize}
\item $R$ is a directed, infinite tree with a finite number of leaves
  and a unique end modelled on the infinite stem \xymatrix{ \bullet
    \ar[r] &\bullet \ar[r] & \bullet \ar[r] &\cdots}, subject to the
  condition that every vertex has one and only one successor (see
  Figure \ref{fig:GradedRoot} for an example),
\begin{figure}[t]
\includegraphics[width=0.25\textwidth]{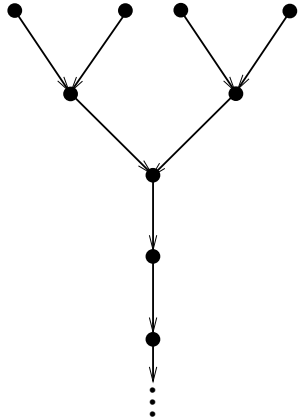}
\caption{A graded root}
\label{fig:GradedRoot} 
\end{figure}
\item $w\colon R \to \Q$ is a function associating to each vertex $x$
  of $R$ a rational number $w(x)\in \Q$ such that $w(x_2)=w(x_1)-2$ if
  $(x_1, x_2)$ is an edge of $R$.
\end{itemize}
A graded root $(R,w)$ specifies a graded $\Field[U]$-module
$\mathbb{H}^-(R,w)$ as follows: As a graded vector space,
$\mathbb{H}^-(R,w)$ is generated over $\F$ by the vertices of $R$,
graded so that $\gr(x)=w(x)$ for all $x\in R$. Multiplication by $U$
is defined on the set of generators by saying that $U \cdot x =y$ if
for the vertex $x$ of $R$ the vertex $y$ is its successor.

Manolescu and Dai showed in \cite{DM} that the lattice homology
$\mathbb{H}^{-}(R,w_{k})$ corresponding to the graded root $(R, w_k)$
can be represented as the homology of a \emph{model complex} $C(R)$,
which is defined as follows.  Let $C(R)$ be generated (as an
$\mathbb{F}[U]$-module), by
\begin{itemize}
\item the leaves  $\{ v_{l} \}$ of $R$ (with grading  $w(v_{l})$), called
the \emph{even generators},  and 
\item by the \emph{odd generators}
 defined as follows: for a vertex $a$ of 
  $R$ with valency greater than 2, let $V=\{v_{1},\cdots,v_{n}\}$ be
  the set of all vertices of $R$ satisfying $Uv=a$.  Then we take the
  formal sums $v_{1}-v_{2},\cdots,v_{n-1}-v_{n}$ of vertices of $R$ as
  generators of $C(R)$ at degree $w(a)+1$.
\end{itemize}
The $U$-action in this module lowers degree by 2.  The differential
$\partial$ on $C(R)$ vanishes on all even generator $v$, while for an
odd generator $a$ the definition is slightly more complicated.  Let
$v$ and $w$ be two even generators such that $a=U^{m}v-U^{n}w$ as
formal sums of vertices of $R$, for some nonnegative integers $m$ and
$n$. Then set $\partial(a)=U^{m+1}v-U^{n+1}w$.  In this case we say
that the odd generator $a$ is an \emph{angle} between the even
generators $v$ and $w$. (For pictorial descriptions of $C(R)$, see
\cite{DM}.)

A negative definite plumbing tree $\Gamma$ (with associated plumbed
four-manifold $X_{\Gamma}$) and a characteristic vector $k$ of
$H^2(X_\Gamma, \Z)$ determines a graded root $(R_\Gamma,w_k)$ as
follows. Let $L$ be the non-compact $1$-dimensional $CW$-complex
having the points of $H_2(X_\Gamma, \Z)=\Z^{\vert \Gamma\vert }$ as
$0$-cells and a $1$-cell connecting two vertices $\ell, \ell' \in
H_2(X_{\Gamma}; \Z)$ if $\ell'=\ell+ v$ for some $v\in \Gamma$. The
characteristic vector $k \in
H^2(X_\Gamma,\Z)={\Hom}(H_2(X_\Gamma),\Z)$ determines a quadratic
function through the formula
\begin{equation}\label{eq:Quadratic}
\chi_k(\ell)=-\frac{1}{2}(k(\ell)+\ell^2). 
\end{equation}
For each $n \in \Z$ let $S_n$ be the set of connected components of
the subcomplex of $L$ spanned by the vertices satisfying $\chi_k \leq
n$. We define $R_\Gamma$ to be the graph with vertex set
$\bigcup_{n\geq 0} S_n$, in which two vertices $x_1$ and $x_2$ are
connected by a directed edge if and only if the elements $x_1\in
S_n$ and $x_2 \in S_{n+1}$ (corresponding to components of the
sublevel sets $\chi _k \leq n$ and $\chi _k \leq n+1$, respectively)
satisfy $x_1 \subset x_2$.  We define $w_k \colon R_\Gamma \to \Q$ 
for $ x \in S_n $ by
the formula
\[
w_k(x)= \frac{k^2+|\Gamma|}{4} -2n. 
\] 
A negative definite plumbing graph $\Gamma$ is \defin{rational} if it
is the resolution graph of a rational singularity, {\rm{i.e.}} a
singularity with geometric genus $p_g=0$. Note that according to
\cite[Theorem 1.3]{NemethiLO} a negative definite graph is rational if
and only if the boundary $Y_\Gamma$ of the associated four-dimensional
plumbing $X_{\Gamma}$ is a Heegaard Floer $L$-space. We say that
$\Gamma$ is \defin{almost-rational} if there is a vertex $v$ of
$\Gamma$ on which we can change the weight in such a way that the
result is rational.

\begin{thm}(\cite{OS20})
\label{isom}
Let $\Gamma$ be an almost-rational graph, $\s \in \Spinc(Y_\Gamma)$ a
$\spinc$ structure on $Y_\Gamma$, and $k$ a characteristic vector of
the intersection lattice of $X_\Gamma$ representing a $\spinc$
structure that restricts to $\s$ on the boundary $\partial
X_\Gamma=Y_\Gamma$. Then there is an isomorphism
$\mathbb{H}^-(R_\Gamma, w_k)\simeq \HFm(Y_\Gamma, \s)$ of graded
$\F[U]$-modules. \qed
\end{thm}

\subsection{Torus knots again}
With this preparation in place, we now return to the computation of
invariants of torus knots. We start by describing a plumbing
presentation of $\Sigma (T_{p,q})$, where the covering involution
$\tau$ is also visible.

\begin{lem}
\label{lem:PlumbingGraphForTorus}
Suppose that $(p,q)=1$ and $pq$ even.
The plumbing graph $\Gamma=\Gamma _{p,q}$ presenting the
double branched cover $\Sigma (T_{p,q})$ of $S^3$ branched along $T_{p,q}$
can be assumed to have the following properties:
\begin{itemize}
\item $\Gamma$ is star-shaped with three legs $L_{fix},L_{1},L_{2}$
going out from the central vertex $c$.
\item The coefficients on $L_{1}$ and $L_{2}$ are the same.
\item The covering involution on $\Sigma (T_{p,q})$ can be modelled by a map
on $\Gamma$ which fixes $L_{fix}$ and $c$,
and flips $L_{1}$ and $L_{2}$.
\item The covering involution on $\Sigma(K)$ extends smoothly to the
  plumbed 4-manifold $X_{\Gamma}$, which is compatible with the above
  involution on $\Gamma$.
\end{itemize}
\end{lem}
\begin{proof}
Recall that the double branched cover $\Sigma (T_{p,q})$ is equal to
the link of the hypersurface singularity $z^2=x^p+y^q$, hence $\Sigma
(T_{p,q})$ can be presented as a plumbed manifold along the resolution
dual graph of the above singularity. This graph can be easily
determined by computing first the embedded resolution of the curve
singularity $x^p+y^q=0$, and then using a simple algorithm (described,
for example in \cite[Section~7.2]{GS}) for computing the resolution
graph of the singularity.  The embedded resolution of $T_{p,q}$ gives
a linear graph, where the single $(-1)$-curve is intersected by the
proper transform of the knot, and the multiplicities at the two ends
are $p$ and $q$, respectively (see \cite{EN}). Now
\cite[Lemma~7.2.8]{GS} shows that if say $p$ is even, then the curves
between the leaf with multiplicity $p$ and the $(-1)$-curve
intersecting the proper transform have all even multiplicities.
The algorithm described in \cite[Section~7.2]{GS} then provides the
resolution graph, together with the information about the covering
transformation, satisfying the properties listed in the lemma.
\end{proof}

\begin{lem}
\label{lem:IdentifOfJforpqeven}
Suppose that $\Gamma$ is a plumbing graph as in
Lemma~\ref{lem:PlumbingGraphForTorus} and the action of $\tau$ exchanges the
two legs. Then the associated reduced  connected homology vanishes.
\end{lem}

\begin{proof}
This information will be sufficient to identify the $\Z/2\Z$-action on
the associated graded root $R$. Recall that in computing a graded
root, we have to choose a characteristic vector $k\in H^2 (X_{\Gamma};
\Z)$, so that we can work with the induced weight function
$\chi_{k}$. Here, we choose $k$ to be the canonical
characteristic vector, which is given by $k(v)=-2-v^2$, 
where $v^2$ denotes the self-intersection of the vertex $v$ in the
plumbing graph.  Then $k$ is clearly $\Z/2\Z$-invariant, so the
induced spin$^c$ structure $[k]|_{\partial
  W_{\Gamma}}\in{\rm{Spin}}^{c}(\Sigma(K))$ on the boundary is also
$\Z/2\Z$-invariant.
Recall that the $\Z/2\Z$-action on the set of spin$^c$
structures on $\Sigma(K)$ leaves the unique spin structure
$\mathfrak{s}_{0}$ invariant. Actually, even more is true:
$\mathfrak{s}_{0}$ is the unique fixed point of the action, as shown
in \cite[Page~1378]{grigsby2006knot} and
\cite[Remark~3.4]{levine2008computing}. Hence the lattice homology of
$\Gamma$ with respect to $k$ computes $\HFm
(\Sigma(K),\mathfrak{s}_{0})$.  Since we are using a
$\Z/2\Z$-invariant characteristic vector, the associated graded root
$R_{k}$, computed using the weight function $\chi_{k}$, admits a
$\Z/2\Z$-action which acts by permuting its vertices. In conclusion,
$R_{k}$ is a symmetric graded root (in the terminology of 
\cite[Definition~2.11]{DM}).  

Next, we claim that there exists a $\Z/2\Z$-invariant vertex
$v\in V(R_k)$ with  minimal $\chi_{k}$-value. 
To prove this, we have to find a lattice point
$x=\sum x_{v} \cdot v \in \Z^{\vert \Gamma \vert }$ which
satisfies the following properties:
\begin{itemize}
\item $x$ is $\Z/2\Z$-invariant, i.e. the coefficients of $x$
  on the arm $L_{1}$ and the coefficients of $x$ on $L_{2}$ are the
  same.
\item $\chi_{k}(x)\le \chi_{k}(y)$ for any other lattice point $y=\sum_{v\in
  V(\Gamma)}y_{v}\cdot v$.
\end{itemize}
Once we have found such a lattice points, the claim about the vertex
$v$ can be proved using the following argument. Recall
that vertices of $R_{k}$ are components of sublevel sets of
$\chi_{k}$. The $\Z/2
\Z$-action on $R_{k}$ permutes the components of sublevel sets, hence
if we denote the component of the minimally-weighted sublevel set
which contains the invariant lattice point $x$ by $C$, then $C$ is
fixed by the $\Z/2 \Z$-action.  

To see the existence of $x$ with the above properties, we first choose
any lattice point $x_{0}=\sum_{v\in V(\Gamma}(x_{0})_{v}\cdot v$ such
that $\chi_{k}(x_{0})$ is minimal. Then we can write $\chi_{k}(x_{0})$ as
\[
\chi_{k}(x_{0})=\chi_{k}(x^{fix}_{0})+Q_{1}(x_{0})+Q_{2}(x_{0}),
\]
 where $x^{fix}_{0} = \sum_{v\in L_{fix}\cup c} x_v \cdot v$ is the
 "fixed part" $x_{0}$ and $Q_{1},Q_{2}$ are functions defined on
 $\Z^{\vert \Gamma \vert }$ using the formula
\[
-2Q_{i}(x) = \sum_{v\in L_{i}} (x^{2}_{v} \cdot e_{v} + x_{v} k_{v}) + \sum_{\text{edge }(v_{1} v_{2})\text{ in }L_{i}} x_{v_{1}} x_{v_{2}} + x_{v^{i}_{c}} x_{c}
\]
for $i=1,2$. (Here, we have denoted the vertex in $L_{i}$ which is connected
to the central vertex $c$ as $v_{c}^{i}$.)

Assume, without loss of generality, that $Q_{1}(x_{0})\le Q_{2}(x_{0})$,
and consider the lattice point $x^{\prime}$ defined as
\[
x^{\prime}=x^{fix}_{0} + \sum_{v\in V(L_{1})}x_{v}\cdot(v+\sigma(v)),
\]
 where $\sigma$ denotes the $\Z/2\Z$-action on $\Gamma$. Then we have 
\[
\chi_{k}(x^{\prime})=\chi_{k}(x^{fix}_{0})+2Q_{1}(x_{0})\leq 
\chi_{k}(x_{0})+Q_{1}(x_{0})+ Q_{2}(x_{0})= \chi_{k}(x_{0}).
\]
Since we assumed that $\chi_{k}(x_{0})$ is minimal among all lattice
points, we get $\chi_{k}(x^{\prime})=\chi_{k}(x_{0})$, implying
$x^{\prime}$ satisfies the desired properties.

Now we claim that our symmetric graded root $R=R_{k}$ is locally
equivalent to another symmetric graded root $R^{\prime}$, where the
$\Z/2\Z$-action on $R^{\prime}$ is trivial. This can be verified by
induction on the number $n_{R}$ of non-$\Z/2\Z$-invariant leaves of
$R$.  In this part of the proof, in the notation we will confuse
graded roots with their associated model complexes.  For example, when
we write that two given symmetric graded roots are locally equvialent,
we will actually mean that their associated model chain complexes are
locally equivalent.

The base case is simple: if $n_{R}=0$, then the
$\Z/2\Z$-action on $R$ is already trivial, so we are done. In
the general case, choose a non-$\Z/2\Z$-invariant leaf $v$ of
$R$. Since $R$ carries a $\Z/2\Z$-invariant leaf in its
top-degree level, we can always find a $\Z/2\Z$-invariant
vertex $x$ of $R$ which lies in the same grading as $v$ does. Denote
the angle between the infinite monotone path starting at $v$ and at $x$
by $\alpha$. Then we define $R_{1}$ as
the graded root associated to the model complex
we get from the model complex of $R$ by deleting $v, \sigma (v)$ and 
the angles $\alpha, \sigma (\alpha )$.
Define a map $F$ from the associated model complex of $R$ to that of
$R_{1}$ as follows.
\begin{itemize}
\item $F(v)=F(\sigma (v))=x$ and $F(w)=w$ for any leaf $w\ne v$.
\item $F(\alpha)=F(\sigma(\alpha))=0$ and $F(\beta)=\beta$
 for any angle $\beta\ne\alpha$.
\item Extend this map $\Field [U]$-linearly to the model chain complex
  of $R$.
\end{itemize}
This 
map $F$ is obviously $\Field [U]$-linear and $\Z/2\Z$-equivariant, so
if we only prove that $F$ is a chain map, it would automatically be a
local equivalence to its image. To check that $F$ is a chain map, it
suffices to check that $\partial(F(v))=F(\partial v)$ and
$\partial(F(\alpha))=F(\partial\alpha)$ by linearity and
equivariance. Indeed, 
\begin{itemize}
\item 
$\partial(F(v))=\partial(\text{even generator})=0$ and 
$F(\partial v)=F(0)=0$,
\item $\partial(F(\alpha))=\partial(0)=0$ and
  $F(\partial\alpha)=F(v+x)=F(v)+F(x)=x+x=0$.
\end{itemize}
Consequetly $F$ is a local equivalence. This implies that $R$ is
locally equivalent to $R_{1}$, and the number of
non-$\Z/2/\Z$-invariant leaves of $R_{1}$ is strictly smaller than the
number of non-$\Z/2\Z$-invariant leaves of $R$. Thus, by induction, we
deduce that $R$ is locally equivalent to a symmetric graded root
$R^{\prime}$ whose $\Z/2\Z$-action is trivial. This gives us the
equality $\overline{\delta}(K)=\underline{\delta}(K)$, which by
Proposition~\ref{prop:Condition} then implies the claim.
\end{proof}

\begin{proof}[Proof of Theorem~\ref{thm:TorusVanishing}]
Proposition~\ref{prop:TorusOdd} in the case $pq$ odd, and the
combination of Lemmas~\ref{lem:PlumbingGraphForTorus} and
\ref{lem:IdentifOfJforpqeven} show that the covering transformation on
$\Sigma (T_{p,q})$ is homotopic to $\id$. The rest of the statement
then follows as for the case of alternating knots.
\end{proof}

\section{Arborescent and Montesinos knots}
\label{sec:Montesinos}
A plumbing tree $\Gamma$ is a tree whose vertices are
labelled by integers (see $\Gamma$ on Figure~\ref{fig:WeightedGraph} below).
\begin{figure}[t]
\input{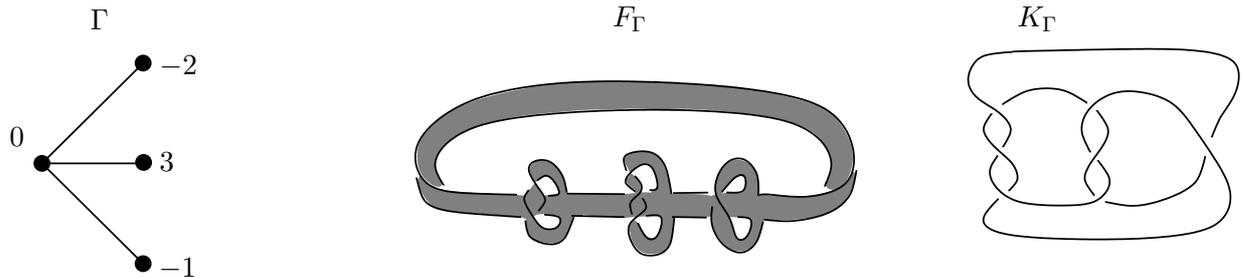}
\caption{The plumbing graph $\Gamma$ on the left
determines a surface $F_{\Gamma}$ (shown in the middle) with 
boundary $K_{\Gamma}$.}
\label{fig:WeightedGraph} 
\end{figure}
Above we associated a four-manifold $X_{\Gamma}$ (and its
three-dimensional boundary $Y_{\Gamma}$) to a plumbing tree
$\Gamma$. A variant of this construction associates a surface
$F_\Gamma\subset S^3$ to $\Gamma$: for every vertex we
consider an annulus or a M\"obius band, given by introducing
half-twists dictated by the label of the vertex, and plumb (or
Murasugi sum) these annuli and M\"obius bands together according to
$\Gamma$.  (See Figure~\ref{fig:WeightedGraph} for a simple example.)
The boundary of $F_\Gamma$ specifies a link $K_\Gamma=\partial
F_\Gamma$, called an \defin{arborescent link} associated to $\Gamma$.
\begin{rem}
Notice that the link is not determined uniquely by the graph, since at
vertices of higher valency we need to determine an order for the edges
when considering a planar presentation, and the link might depend on
this choice. In addition, the location of the plumbing region relative
to the twists also might influence the resulting link. With slightly
more information (see Gabai's introductory work in \cite{Gabai})
attached to the tree this procedure can be made unique, though.
\end{rem}

  The resulting $K_{\Gamma}$
is called a \defin{Montesinos link} if the tree $\Gamma$ is
star-shaped (i.e., has at most one vertex with valency more than 2);
it is a \defin{pretzel link} if $\Gamma$ is star-shaped with all legs
of length one, and is 2-bridge if $\Gamma $ is linear (i.e., all
valencies are either 1 or 2).

The construction of the four-manifold and the knot associated to
$\Gamma$ is connected by the fact that (by repeated application of
Montesinos' trick) the double branched cover $\Sigma(K_\Gamma)$
of an arborescent link $K_\Gamma$ is homeomorphic to the boundary
$Y_\Gamma$ of the four-dimensional plumbing $X_\Gamma$ associated to
$\Gamma$. Indeed, $X_{\Gamma}$ can be presented as the double branched cover
of $D^4$ branched along the surface we get by pushing the interior of the
surface $F_{\Gamma}$  of the above pluming into $D^4$.

Next we give an algorithm for computing the connected Floer homology
group $\HFB_{\rm{conn}}(K)$ of an arborescent knot based on lattice
homology.  Recall that the graded module $\mathbb{H}^-(R,w)$
associated to a graded root $(R,w)$ is the homology of a canonical
free, finitely generated chain complex (the model complex) $(C(R),
\partial)$ over the polynomial ring $\F[U]$. The following was
observed in \cite[Remark 4.3]{DM}.

\begin{lem}
\label{algebraicfundation}
If $(R,w)$ is a graded root, then a grading preserving homomorphism
$f\colon \mathbb{H}^-(R,w) \to \mathbb{H}^-(R,w)$ lifts, uniquely up
to chain homotopy, to a grading preserving chain map $f_\#\colon
C(R) \to C(R)$.  In addition, if $C$ is a free chain complex
with $H_*(C)\simeq \mathbb{H}^-(R,w)$ then $C\simeq C(R)$. \qed
\end{lem} 

If $k$ is a characteristic vector which restricts to a self-conjugate
$\spinc$ structure $\s_0 \in \Spinc(Y_\Gamma)$, then the graded root
$(R_\Gamma, w_k)$ comes with an involution $J\colon
\mathbb{H}^-(R,w)\to \mathbb{H}^-(R,w)$ \cite[Section 2.3]{DM}. This
is the map induced on $\mathbb{H}^-(R,w)$ by 
\begin{equation}\label{eq:invo}
\ell \mapsto -\ell-\frac{1}{2}PD(k),
\end{equation}
for $\ell \in H_2(X_\Gamma,\Z)$. We denote its lift to $C(R)$ by
$J_\#\colon C(R_\Gamma) \to C(R_\Gamma)$.

\begin{thm}\label{montesinoscomp}
Let $K=K_\Gamma$ be an arborescent knot associated to an
almost-rational plumbing tree $\Gamma$. Then there exists a chain
homotopy equivalence $(\CFm(\Sigma(K),\s_0),
\tau_\#)\simeq(C(R_\Gamma), J_\#)$ of $\iota$-complexes.
\end{thm} 
\begin{proof}
Let $k_0\in H^2(X_\Gamma, \Z)$ be a characteristic vector of the
intersection lattice of $X_\Gamma$ which restricts to $\s_0$ on the
boundary. Theorem~\ref{isom} provides an isomorphism
\[H_*(\CFm(\Sigma(K),\s_0))=\HFm(\Sigma(K),\s_0)\simeq \mathbb{H}^-(R_\Gamma,w_{k_0})=H_*(C(R_\Gamma)) \ .\] 
As a consequence of Lemma \ref{algebraicfundation}, if we prove that
the push-forward of $J\colon \mathbb{H}^-(R_\Gamma,w_{k_0}) \to
\mathbb{H}^-(R_\Gamma,w_k)$ through this isomorphism agrees with
$\tau_*\colon \HFm(\Sigma(K),\s_0) \to \HFm(\Sigma(K),\s_0)$, we are
done.

Denote by $W_\Gamma$ the cobordism $S^3 \to Y_\Gamma$ obtained from
$X_\Gamma$ by removing the interior of a small ball $D^4 \subset
X_\Gamma$. Unravelling the definition of the isomorphism of Theorem~\ref{isom} as it was done
in \cite[Theorem 3.1]{HM}, the claim boils down to the identity
\[\tau_\# \circ F_{W_\Gamma, k}^-=F_{W_\Gamma, -k}^-  \ ,\]
where $k$ is a  characteristic vector which restricts to $\s_0$ on the
boundary. According to \cite{Saveliev} the covering involution $\tau
\colon Y_\Gamma\to Y_\Gamma$ extends over $X_\Gamma$ to the complex
conjugation $T\colon X_\Gamma \to X_\Gamma$. Since $T$ acts on
$\spinc$ structures as $\spinc$ conjugation,
Lemma~\ref{lem:naturalitylemma}  implies the
claimed identity.
\end{proof}

By applying similar arguments of \cite[Section 5]{DM} we get the
following results.

\begin{cor}
Let $K=K_\Gamma$ be an arborescent knot associated to  an
almost-rational plumbing tree $\Gamma$. Then there is an isomorphism
of graded $\F[U]$-modules $\HFB(K) \simeq \Ker \, (1+J)[-1]\oplus
{\rm{CoKer}}\, (1+J)$. Under this isomorphism the action of $Q$ on
$\HFB(K)$ is given by the projection $\Ker \, (1+J) \to
\Ker \, (1+J)/\Imm(1+J)\subset {\rm{CoKer}}\, (1+J)$. \qed
\end{cor}

\begin{cor}
Let $K=K_\Gamma$ be an arborescent knot associated to an
almost-rational plumbing tree $\Gamma$. Then $\deltaover(K)=\delta(K)$
and $\deltaunder(K)=-\sigma(K)/4$, where $\sigma(K)$ denotes the
signature of $K$.
\end{cor}
\begin{proof}
Following the argument of \cite[Theorem 1.2]{HM} we can identify
$\deltaover(K)$ with the Ozsv\'ath-Szab\'o correction term of the
double branched cover $d(\Sigma(K), \s_0)=\delta(K)$. Furthermore,
$\deltaunder(K)=-2 \cdot \overline{\mu}(\Gamma, \s_0)$ where
$\overline{\mu}(\Gamma, \s_0)$ denotes the Neumann-Siebenmann
$\overline{\mu}$-invariant of the plumbing $\Gamma$ in the spin
structure $\s_0$. On the other hand, according to \cite[Theorem
  5]{SavelievMubar} we have $\sigma(K)=8\cdot \overline{\mu}(\Gamma,
\s_0)$ thus $\deltaunder(K)=-\sigma(K)/4$.
\end{proof}

Note that since a star-shaped plumbing tree $\Gamma$ is
almost-rational, for Montesinos knots the assumption of the above
results modifies to demand that the intersection form of $X_{\Gamma}$
(equivalently, the inertia matrix of the plumbing graph $\Gamma$) is
negative definite. This method of approaching computability questions
will be utilized in the next section.

The connected group associated to the $\iota$-complex $(C(R), J_\#)$
of a graded root $(R,w)$ can be easily computed. Given a vertex $v \in
R$ denote by $C(v)$ the set of all leaves of $R$ that are connected to
$v$ by an oriented path. We construct a subset $S$ of the leaves of
$R$ by the following algorithm.  Let $v_0$ denote the $J$-invariant
vertex $v_0$ of $R$ with highest weight.  If $C(v_0)$ consists of only
one vertex, we add it to $S$; otherwise, we can find a pair $\{v,
Jv\}$ in $C(v_0)$ and in this case we add both $v$ and $Jv$ to
$S$. Next consider the path $\gamma$ connecting $v_0$ to
infinity. Take $v_1 \in R$ to be the first vertex along $\gamma$ for
which $C(v_0) \subsetneq C(v_1)$. If $C(v_1)$ contains a pair of
leaves $\{v, Jv\}$ with weight larger than the weight of any leaf in
$S$ then we choose one such pair with largest possible weight and we
add it to $S$. By keep iterating this procedure until $\gamma$ merges
with the long stem we end up with a set $S$ of distiguished leaves. We
denote by $M$ the smallest graded root $M \subset R$ containing
$S$. According to \cite[Proposition 7.5]{HHL} we have that
$\HFB_\text{conn}(K)=\mathbb{H}^-(M,w|_M)$.  The resulting $M$ is the
\emph{monotone subroot} of $(R, w)$.

\section{Some independence results}
\label{sec:IndepResults}
The following observation is a simple corollary of the concluding
statements of Section~\ref{sec:ConcordanceInv}.

\begin{cor}
\label{cor:generate}
Suppose that ${\mathcal {F}}=\{ K_i\mid i\in I\}$ is a family of knots
with $\HFB_{\rm{red\kotojel conn}}(K)=0$ for all $K\in {\mathcal
  {F}}$.  If $\langle {\mathcal {F}}\rangle $ denotes the subgroup of
${\mathcal {C}}$ generated the equivalences of $K_i\in {\mathcal
  {F}}$, then $\HFB_{\rm{red\kotojel conn}}(K)=0$ for all $K\in
\langle {\mathcal {F}}\rangle$.
\end{cor}
\begin{proof}
Lemmas~\ref{lem:mirror} and \ref{lem:Sums}, together with the
equivalence provided by Proposition~\ref{prop:Condition} imply the
result.
\end{proof}

Recall the definition of the subgroups ${\mathcal {QA}}$ and $
{\mathcal {T}}$ (generated by quasi-alternating and torus knots,
respectively) of the smooth concordance group ${\mathcal {C}}$ from
Section~\ref{sec:intro}.
\begin{prop}
\label{prop:TotalVanishing}
For $[K]\in {\mathcal {QA}}+{\mathcal {T}}$ we have that 
$\HFB _{\rm{red\kotojel conn}}(K)=0$.
\end{prop}
\begin{proof}
In Section~\ref{sec:Vanishing} we showed that $\HFB _{\rm{red\kotojel
    conn}} (K)=0$ once $K$ is either a quasi-alternating knot or a
torus knot. Application of Corollary~\ref{cor:generate} concludes
the argument.
\end{proof}

Based on this proposition, nonvanishing results 
for $\HFB _{\rm{red\kotojel conn}}$
then immediately imply nonvanishing in the quotient group
${\mathcal {C}}/({\mathcal {QA}}+{\mathcal {T}})$. 

Recall that in \cite{LidmanMoore} Lidman and Moore showed that a
pretzel knot $P(a_1, \dots, a_k)$ is an $L$-space knot (\text{i.e.} it
has $L$-space surgeries) if and only if it is a torus knot of the form
$T_{2,2n+1}$, or a pretzel knot of the form $P(-2,3,q)$ for some
$q\geq 7$ odd.  In \cite{Wang} the question of which linear
combinations of $L$-space knots is concordant to a linear combination
of algebraic knots was studied. In \cite[Theorem 1.1]{alfieri1} the first author showed that
pretzel knots of the form $P (-2, 3, q)$ are not concordant to positive sums of
algebraic knots. Note that for these knots the obstruction found in
\cite[Corollary 3.5]{Wang} vanishes. By computing connected  Floer homology now we will prove
Theorem~\ref{thm:pretzels}.

Let $C=(C_0, \partial _0, \iota _0)$ denote the $\iota$-complex where 
$C_0$ (as an $\Field [U]$-module) is generated by three
generators $a,b,c$, the boundary map  $\partial _0$ is given by 
$\partial_0 a=\partial_0 b=0$ and $\partial_0 c=Ua+Ub$ and $\iota _0(a)=b, 
\iota _0(b)=a, \iota _0(c)=c$. Define the grading by $\gr (a)=\gr (b)=0$ 
and $\gr (c)=-1$, and denote by $C[r]$ the same chain complex with grading
shifted by $r\in \Q$. (The same chain complex has been denoted
by ${\mathcal {C}}_1$ in \cite[Subsection~7.1]{HHL}.)

\begin{lem}
\label{lem:CalcForPretzel}
Let $q\ge7$ be any odd integer. Then 
$(\CFm (\Sigma  (2,-3,-q), \s _0), \inv _\#)$ is locally equivalent to
$C[r]$ for some $r$. In particular, $\HFB _{\rm{red\kotojel conn}}
(\Sigma (2, -3, -q))=\Field$.
\end{lem}
\begin{proof}
The pretzel knot $-P(-2,3,q)=P(2,-3,-q)$ is associated to the negative
definite star-shaped (hence almost-rational) plumbing tree $\Gamma _q$
shown in Figure~\ref{fig:23q}(a).  Let us label the vertices of
$\Gamma_{q}$ by $c,v_{1},v_{2},v_{3}$ as shown by
Figure~\ref{fig:23q}(a).
\begin{figure}[t]
\input{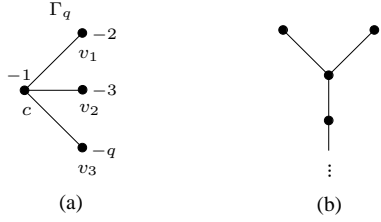}
\caption{On the left we give the plumbing graph $\Gamma _q$ giving
the pretzel knot $P(2,-3,-q)$. On the right the monotone subroot of
the associated graded root is shown.}
\label{fig:23q} 
\end{figure}
(Notice that by successively blowing down $c, v_1$ and $v_2$ we can see
that the three-manifold $Y_{\Gamma _q}$ defined by the plumbing is
the result of $(-q+6)$-surgery on the right-handed trefoil knot.)

In determining the connected Floer homology of these pretzel knots we
will appeal to the computational scheme through lattice homology (as
outlined in Section~\ref{sec:Montesinos} for appropriate arborescent
knots). Recall that the canonical characteristic vector $K_{q}\in H^2
(X_{\Gamma _q}; \Z )$ is the class satisfying $K_q(v)=-2-v^2$ for each
vertex (where a vertex $v$ is viewed as an element of $H_2(X_{\Gamma
  _q}; \Z )$ and $v^2$ denotes the weight on the vertex $v$, which is
equal to self-intersection of the corresponding homology class).  In
particular, $K_q(c)=-1, K_q(v_1)=0, K_q(v_2)=1$ and $K_q(v_3)=q-2$.
Every characteristic cohomology element determines a spin$^c$
structure on the boundary three-manifold; by following the blow-down
sequence described above with $K_q$ we see that it induces the unqiue
spin structure $\s _0$ on the boundary. For an element $x= \alpha
c+\beta v_{1}+\gamma v_{2}+\delta v_{3} \in H_2 (X _{\Gamma _q}; \Z )$
the quadratic function $\chi _{K_{q}}$ of
Equation~\eqref{eq:Quadratic}  is given by
\[
2\chi_{K_{q}}(\alpha c+\beta v_{1}+\gamma v_{2}+\delta v_{3})
=\alpha^{2}+2\beta^{2}+3\gamma^{2}+q\delta^{2}-
2\alpha(\beta+\gamma+\delta)+\alpha-\gamma-(q-2)\delta.
\]

It is not hard to see that $\chi _{K_q}(x) \geq 0$ for any $x\in H_2
(X_{\Gamma _q}; \Z )$. In order to determine the monotone subroot
$M_q$ of the graded root $(R_{K_q}, w_{K_q})$ it will be sufficient to
understand $S_0$ and $S_1$ (in the terminology of
Subsection~\ref{ssec:lattice}, i.e., $S_n=\{ \chi _{K_q}\leq n\}$).
For $q=7$ this calculation has been performed in
\cite[Subsection~3.2]{OS20}, where it has been shown that $\{ \chi
_{K_7}\leq 0\}$ has two components, while $\{ \chi _{K_7}\leq 1\}$ has
a single component. This shows that $M_7$ is of the form given by 
Figure~\ref{fig:23q}(b).

Notice, however, that for $x= \alpha c+\beta v_{1}+\gamma v_{2}+\delta
v_{3}$ we have $2\chi _{K_q}(x) =2\chi _{K_7}(x)+(q-7)(\delta
^2-\delta )$.  Since $\chi _{K_q}(x)\geq 0$ for all $x$, we have that
$S_0$ and $S_1$ have the same cardinality for all $q\geq 7$, hence the
monotone subroot will be the same for all $q$.

The explicit description of the two vectors representing $S_0$ (as
given in \cite{OS20}) and the description of the involution of
Equation~\eqref{eq:invo} now identifies the action of $\inv _{\#}$,
verifying the claim about the action of $\inv _\#$. Then the definition
of $\HFB_{\rm{conn}}(K)$ shows that in this case it is isomorphic to 
$\HFm (\Sigma (K), \s _0)$, and hence $\HFB _{\rm{red\kotojel conn}}(K)=
\Field$, completing the proof.  
\end{proof}
\begin{rem}
Indeed, $\HFa (\Sigma (P(2,-3,-q)))$ can be easily computed using the
surgery exact triangle (by viewing this three-manifold
as surgery along the trefoil knot).
This computation shows that $\dim \HFa
(\Sigma (P(2,-3,-q)), \s _0)=3$ (and all other spin$^c$ structures
have one-dimensional invariant), determining $\HFa (\Sigma
(P(2,-3,-q)), \s )$ as well, and showing that the local equivalence
claimed in Lemma~\ref{lem:CalcForPretzel} is, indeed, a chain homotopy
equivalence. Also, the gradings can be easily determined by computing
$K_q^2$. The detailed calculation in the proof of
Lemma~\ref{lem:CalcForPretzel} is crucial in the
identification of the involution $\inv _\#$.
\end{rem}

\begin{proof}[Proof of Theorem~\ref{thm:pretzels}]
Suppose that $K$ is the connected sum of pretzel knots $P(2,-3,-q)$
for some $q$'s (all with $q\geq 7$).  By \cite[Proposition~7.1]{HHL}
the tensor product of these $\iota$-complexes have nonvanishing
connected homology, hence the combination of this nonvanishing result
with Proposition~\ref{prop:TotalVanishing} implies that $[K]\in
{\mathcal {C}}/({\mathcal {QA}}+{\mathcal {T}})$ is nontrivial.  In
the statement of the theorem we considered mirrors of the knots
encountered above; since $\HFB_{\rm{red\kotojel conn}}(-K)$ and
$\HFB_{\rm{red\kotojel conn}}(K)$, as well as $[K]$ and $[-K]$ in
${\mathcal {C}}/({\mathcal {QA}}+{\mathcal {T}})$ are trivial at the
same time, the proof is complete.
\end{proof}

Next we verify Theorem~\ref{thm:PretzelLinIndep} from
Section~\ref{sec:intro}.  For a knot $K \subset S^3$ define
$\omega(K)=\min\{n \geq 0 \ | \ U^n\cdot \HFB_{\rm{red\kotojel
    conn}}(K)=0 \}$.

\begin{proof}[Proof of Theorem~\ref{thm:PretzelLinIndep}]
The double branched cover of the pretzel knot $K_q=P(4q+3,-2q-1,
4q+1)$ can be expressed as boundary of the negative definite plumbing
$\Gamma_q$ of Figure~\ref{fig:PretzelPlumbing}.
\begin{figure}[t]
\input{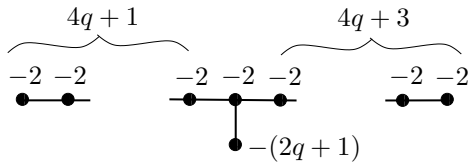}
\caption{The plumbing graph $\Gamma _q$.}
\label{fig:PretzelPlumbing} 
\end{figure}
The associated graded root $R_q$ was partially computed in
\cite{HKL}. According to \cite{HKL} the top part of $R_q$, describing
the truncated Heegaard Floer group
$\HFm_{\geq\delta(K_q)-2q}(\Sigma(K_q))$, looks like the oriented
graph of Figure~\ref{fig:gradedroot}. 
\begin{figure}[t]
\input{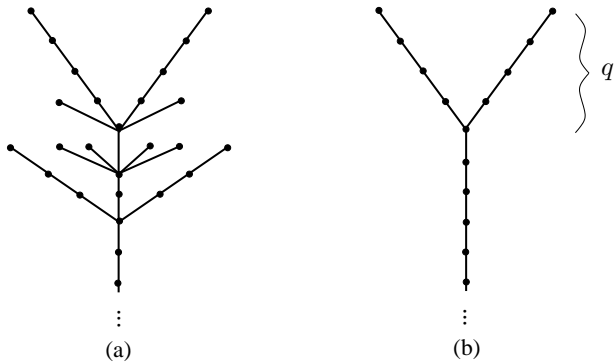}
\caption{Schematic picture of the graded root associated to $\Gamma _q$ in (a)
and its monotone subroot in (b).
Notice that there is no element of top degree fixed by the reflection.} 
\label{fig:gradedroot} 
\end{figure}

The monotone sub-root $M_q \subset R_q$ of a
graded root of this form has been determined in \cite[Theorem~6.1]{DM}, 
providing the graded roots of Figure~\ref{fig:gradedroot}(b)
(see also \cite[Figure~4]{DM}), yielding
\[
\text{HFB}_{{\rm {conn}}}^{-}(K_q)= \F[U] \oplus \F[U]/(U^{q}) .
\]  
The application of \cite[Proposition~7.1]{HHL} calculates
$\text{HFB}_{{\rm {conn}}}^{-}(K)$ for a knot $K$ that is a positive linear
combination of the pretzel knots $K_q$:
\[
\text{HFB}_{{\rm {conn}}}^{-}(K_{q_1}\# \dots \# K_{q_s})= 
\F[U] \oplus \bigoplus_{i=1}^s \F[U]/U^{q_i} \ .
\]
In particular, for a knot of this form we have that
\begin{equation}\label{computation}
\omega(K_{q_1}\# \dots \# K_{q_s})=\max_i q_i \ . 
\end{equation} 

Suppose that there is a non-trivial linear relation among the classes
represented by the pretzel knots $K_q$ in
$\mathcal{C}/(\mathcal{T}+\mathcal{QA})$. By grouping the summands of
such a linear relation according to their sign we end up with an
identity of the form
\[K_{a_1}\# \dots \# K_{a_s} = K_{b_1}\# \dots \# K_{b_l} \# P\]
in the knot concordance group $\mathcal{C}$, for some $P \in
\mathcal{T}+\mathcal{QA}$. Without loss of generality we can assume
that the relation is reduced, \text{i.e.} that $a_i \not=b_j$ for each
$i$ and $j$. On the other hand as a consequence of
Equation~\eqref{computation}  we have that
\[\max_i a_i= \omega(K_{a_1}\# \dots \# K_{a_s})=\omega (K_{b_1}\# \dots \# K_{b_l} \# P )= \omega (K_{b_1}\# \dots \# K_{b_l})=\max_j b_j \] 
a contradiction. This shows that the subgroup $\langle [K_q]\mid q\,
{\rm {odd\, and }}\, q\geq 7\rangle\subset {\mathcal {C}}/({\mathcal
  {T}}+{\mathcal {QA}})$ generated by the above pretzel knots is
isomorphic to $\Z ^{\infty}$, concluding the proof.
\end{proof}

\bibliography{coverbiblio} \bibliographystyle{plain}
\end{document}